\documentclass[oneside,11pt,reqno]{amsart}
\usepackage{amssymb,amsmath,amsthm,bbm,enumerate,mdwlist,url,multirow,hyperref,amsthm}
\usepackage[pdftex]{graphicx}
\usepackage[shortlabels]{enumitem}

\theoremstyle{definition}
\newtheorem{definition}{Definition}
\theoremstyle{theorem}
\newtheorem{proposition}[definition]{Proposition}

\newtheorem{corollary}[definition]{Corollary}

\numberwithin{equation}{section}
\theoremstyle{remark}
\newtheorem{remark}[definition]{Remark}
\def\PP{\mathsf P}

\def\EE{\mathsf E}

\def\FF{\mathcal F}

\def\diffusion{\sigma^2}
\def\drift{\gamma_0}
\def\measure{\nu}
\def\supp{\mathrm{supp}}

\def\law{\lambda}
\def\defeq{\buildrel \text{d{}ef}\over =}
\begin{document}
\title{Proper two-sided exits of a L\'evy process}

\author{Matija Vidmar}
\address{Department of Mathematics, University of Ljubljana and Institute of Mathematics, Physics and Mechanics, Slovenia}
\email{matija.vidmar@fmf.uni-lj.si}

\begin{abstract}
It is proved that the two-sided exits of a L\'evy process  are proper, i.e. not a.s. equal to 
their one-sided counterparts, if and only if said process is not a subordinator or the negative of a subordinator. Furthermore, L\'evy processes are characterized, for which the supports of the first exit times from bounded annuli, simultaneously on each of the two events corresponding to exit at the lower and the upper boundary, respectively are unbounded, contain $0$, are equal to $[0,\infty)$. 
\end{abstract}


\keywords{L\'evy process; two-sided exit; support of measure}


\maketitle

\section{Introduction}
Two-sided exits of spectrally one-sided L\'evy processes have been extensively studied, e.g.  \cite[Chapter~VII]{bertoin} \cite[Section~9.46]{sato} \cite[Section~8.2]{kyprianou}, to which is added a great number of scientific papers. Less is known in the general case -- though  integral transforms of many relevant quantities still admit an analytic representation \cite{kandakov}. But the expressions entering these transforms are complicated, and in particular do not lend themselves easily to analysis. The study of the qualitative aspects of the two-sided exit problem, at least for the case of a general L\'evy process, thus appears relevant. 

Such study is clearly connected with that of the distributional properties of the running supremum $\overline{\vert X\vert}$ of the absolute value $\vert X\vert$ of a L\'evy process $X$. However -- by contrast to those of the supremum  process $\overline{X}$ of $X$ itself, e.g. \cite{pecherskii_rogozin,rogozin,doney,chaumont} --, few such properties appear to have been analyzed in general. There are exceptions, e.g. \cite{simon,aurzada}. 

In a minor contribution to this area, the purpose of the present paper is to characterize those L\'evy processes for which the supports of the first exit times from bounded annuli, simultaneously on each of the two events corresponding to exit at the lower and the upper boundary, are respectively non-empty (Proposition~\ref{proposition:non-one-sided-exit}), contain $0$ (Proposition~\ref{proposition:proper-exit-before}), are unbounded (Proposition~\ref{proposition:proper-exit-after}), are equal to $[0,\infty)$ (Corollary~\ref{corollary:proper-exit-full-support}). Propositions~\ref{proposition:more-precise-one} and~\ref{proposition:more-precise-two} give a further description of the cases when, respectively, the second and third of the preceding properties fails, but the first does not. 

In terms of practical relevance note that L\'evy processes are often used to model the risk process of an insurance company \cite[Paragraph~1.3.1 \& Chapter~7]{kyprianou}, or their exponentials are 
used to model the price fluctuations of stocks \cite[Paragraph~2.7.3]{kyprianou}. Thus, for example, it may be useful to know whether or not  (in both cases possibly before or after some time, or in each non-degenerate time interval) (i) an insurer with given initial capital will in fact go bankrupt or its capital will exceed some given level, each with a positive probability; or (ii) a perpetual two-sided barrier option will terminate with a positive probability on each of the two boundaries. 

\section{Setting and notation}
Throughout we will let $X$ be a L\'evy process \cite[Section~1]{sato} on the stochastic basis $(\Omega,\FF,\mathbb{F},\PP)$ (it is assumed then that $X$ is $\mathbb{F}$-adapted) satisfying the standard assumptions, with diffusion coefficient $\diffusion$, L\'evy measure $\measure$ and, when $\int 1\land \vert x\vert\measure(dx)<\infty$, drift $\drift$ \cite[Section~8]{sato}.

\begin{definition}[Two-sided exit times and their laws]
Let $\{a,b\}\subset (0,\infty)$. 
\begin{enumerate}[(i)]
\item For a  c\`adl\`ag path $\omega$ mapping $[0,\infty)$ into $\mathbb{R}$, vanishing at zero, we denote by $T_{a,b}(\omega)$ the first entrance time of $\omega$ into the set $\mathbb{R}\backslash (-b,a)$ (i.e. the first exit time of $\omega$ from $(-b,a)$). 
\item We introduce the measures $\law_{a,b}^+$ and $\law_{a,b}^-$ on $\mathcal{B}([0,\infty))$, so that for $A\in \mathcal{B}([0,\infty))$, $\law_{a,b}^-(A)\defeq\PP(\{T_{a,b}(X)<\infty\}\cap \{X_{T_{a,b}(X)}\leq -b\}\cap \{T_{a,b}(X)\in A\})$ and $\law_{a,b}^+(A)\defeq\PP(\{T_{a,b}(X)<\infty\}\cap \{X_{T_{a,b}(X)}\geq a\}\cap \{T_{a,b}(X)\in A\})$.
\end{enumerate}
\end{definition}
We shall be concerned then with characterizing the pairs $(\diffusion,\measure)$ and, when $\int1\land \vert x\vert\measure(dx)<\infty$, further the drifts $\drift$, under which the measures $\law^\pm_{a,b}$ are non-vanishing (on each non-empty interval of the form $[0,M)$, $[m,\infty)$, respectively $[m,M)$). 

\begin{definition}[Auxiliary notions/notation]
For a time $S:\Omega\to [0,\infty]$ we will call $(X(S+t)-X(S))_{t\geq 0}$  (defined on $\{S<\infty\}$) the incremental process of $X$ after $S$. For a measure $\rho$ on a topological space, $\supp(\rho)$ will be its support. Finally, $a\land b:=\min\{a,b\}$ (when $\{a,b\}\subset [-\infty,+\infty]$): for measurable sets $A$ and $B$ and a measure $\lambda$, $\lambda A\land \lambda B>0$ is thus shorthand for ``$\lambda(A)>0$ and $\lambda(B)>0$.''
\end{definition}

\section{Results}
Now the precise statements follow.

\begin{proposition}\label{proposition:non-one-sided-exit}
$\law_{a,b}^+[0,\infty)\land \law_{a,b}^-[0,\infty)>0$ for some (then all) $\{a,b\}\subset (0,\infty)$, if and only if 
\begin{quote}
$\diffusion>0$; or $\int\vert x\vert\land 1\measure(dx)=\infty$; or $\nu$ charges $(-\infty,0)$ and $(0,\infty)$ both; or else $\nu$ charges only (and does charge) $(0,\infty)$ and $\drift<0$, or $\nu$ charges only (and does charge) $(-\infty,0)$ and $\drift>0$,
\end{quote}
 i.e. if and only if  neither $X$ nor $-X$ is a subordinator. 
\end{proposition}
\begin{proof}
For the last equivalence see \cite[p. 137, Theorem~21.5]{sato}. The condition is clearly necessary. 

Sufficiency. Let $\{a,b\}\subset (0,\infty)$. The condition implies $X$ is not the zero process, so that $\limsup_\infty X=\infty$ or $\liminf_{-\infty}X=-\infty$ a.s. \cite[p. 255, Proposition~37.10]{sato}, and so a.s. $T_{a,b}(X)<\infty$. Suppose furthermore \emph{per absurdum}, and then without loss of generality, that a.s. $X_{T_{a,b}(X)}\geq a$. Let $X^{(0)}\defeq X$,  $T^{(0)}\defeq T_{a,b}(X)$. Then by the strong Markov property \cite[p. 278, Theorem~40.10]{sato} of L\'evy processes, inductively, we would find that a.s. for all $k\in \mathbb{N}_0$ the incremental process $X^{(k+1)}$ of $X$ after $T^{(k)}$ would satisfy $X^{(k+1)}(T^{(k+1)})\geq a$, where $T^{(k+1)}\defeq T_{a,b}(X^{(k+1)})$ would be equal in distribution to $T^{(0)}$ and independent of $\FF_{T^{(k)}}$. In particular, since by the right-continuity of the sample paths $\EE T^{(0)}>0$ (indeed $T^{(0)}>0$ a.s.), and since $(T^{(k)})_{k\in \mathbb{N}_0}$ is an iid sequence, it would follow from the strong law of large numbers, that with probability one $X$ would be $> -b$ at all times. According to \cite[p. 149, Theorem~24.7]{sato} this would only be possible if $\diffusion=0$, $\int\vert x\vert\land 1\measure(dx)<\infty$ with $\nu$ charging only $(0,\infty)$. Then according to the assumed condition we would need to have $\drift<0$, yielding a contradiction with \cite[p. 151, Corollary~24.8]{sato}, which necessitates the infimum of the support of $X_t$ being $\gamma_0 t$, for all $t\in [0,\infty)$, in this case.
\end{proof}

In various subcases, this statement can be made more nuanced. 

\begin{proposition}\label{proposition:proper-exit-before}
The condition that $\law^+_{a,b}[0,M)\land \law^-_{a,b}[0,M)>0$ for all $M\in (0,\infty]$, for some (then all) $\{a,b\}\subset (0,\infty)$, is equivalent to 
\begin{quote}
$\nu$ charges $(-\infty,0)$ and $(0,\infty)$ both; or $\diffusion>0$; or $\int 1\land\vert x\vert\measure(dx)=\infty$.
\end{quote} 
\end{proposition}
\begin{proof}
The condition is necessary. For, if $\diffusion=0$, $\int 1\land \vert x\vert<\infty$ and, say, $\nu$ charges only $(0,\infty)$, then in order that $X$ not have monotone paths, it will need to assume a strictly negative drift, but even then, according to the L\'evy-It\^o decomposition \cite[Section~2.4]{applebaum}, for given $a$ and $b$, $M$ can clearly be chosen so small, that by time $M$, a.s. $X$ can only have left $(-b,a)$ at the upper boundary. 

Sufficiency. The argument is similar as in the proof of the preceding proposition, so we forego explicating some of the details. Let $\{a,b\}\subset (0,\infty)$, $M\in (0,\infty)$. Suppose \emph{per absurdum}, and then without loss of generality, that $X_{T_{a,b}(X)}\geq a$ a.s. on $\{T_{a,b}(X)< M\}$. Let $X^{(0)}\defeq X$,  $T^{(0)}\defeq T_{a,b}(X)$. By the strong Markov property of L\'evy processes, inductively, we find that a.s. for all $k\in \mathbb{N}_0$ the incremental process $X^{(k+1)}$ of $X$ after $T^{(k)}$ satisfies $X^{(k+1)}(T^{(k+1)})\geq a$ on $\{T^{(k+1)}< M\}$, where $T^{(k+1)}\defeq T_{a,b}(X^{(k+1)})$ is equal in distribution to $T^{(0)}$ and independent of $\FF_{T^{(k)}}$. From the strong law of large numbers, it now follows, that with probability one $X$ is $> -b$ on $[0,M)$. But this is only possibly if $\diffusion=0$, $\int\vert x\vert\land 1\measure(dx)<\infty$ with $\nu$ charging only $(0,\infty)$, a contradiction.
\end{proof}
The situation when $X$ satisfies the condition of Proposition~\ref{proposition:non-one-sided-exit} but not that of Proposition~\ref{proposition:proper-exit-before} can (up to the trivial transformation $X\rightarrow -X$) easily be described as follows (in particular, in Proposition~\ref{proposition:proper-exit-before}, we cannot change the qualification ``for all $M\in (0,\infty]$, for some (then all) $\{a,b\}\subset (0,\infty)$'' to ``for some (then all) $M\in (0,\infty]$ and $\{a,b\}\subset (0,\infty)$''):
\begin{proposition}\label{proposition:more-precise-one}
Suppose $\int 1\land\vert x\vert\measure(dx)<\infty$, $\diffusion=0$ and $\nu$ charges only (and does charge) $(-\infty,0)$, finally $\gamma_0>0$. Let furthermore $M\in (0,\infty]$, $\{a,b\}\subset (0,\infty)$. Then $\law^+_{a,b}[0,M)\land \law^-_{a,b}[0,M)>0$, if and only if $a/\gamma_0<M$. 
\end{proposition}
\begin{proof}
The condition is clearly necessary. Sufficiency. In view of Proposition~\ref{proposition:non-one-sided-exit}, we may assume $M<\infty$. If $\measure$ is finite, note that there is $\epsilon \in (0,\infty)$ with $\nu(-\infty,-\epsilon)>0$, and then the desired conclusion follows from the fact that with positive probability the process $X$ will have $\lceil (a+b)/\epsilon\rceil$ many jumps of size $<-\epsilon$ before $a/\gamma_0$, going below $-b$ (but not above $a$) strictly before time $M$, but with positive probability will also not have a jump up to, reaching level $a$ at, time $a/\gamma_0<M$. If $\measure$ is infinite, it follows simply from the fact that the support of the jump part of $X$ will be $(-\infty,0]$ at all strictly positive times \cite[p. 152, Theorem~24.10]{sato}: thus, with a positive probability, at time $a/\gamma_0+(M-a/\gamma_0)/2$, say, the jump part will not have gone below the level $-(((M\gamma_0-a)/2)\land b)$, and hence $X$ not below the level $-b$, but $X$ will of course have gone above the level $a$ by this time; in the same vein, with a positive probability, at time $a/(2\gamma_0)$, say, the jump part of $X$ will have gone below $-(b+a/2)$, and so $X$ below the level $-b$, but $X$ will not have gone above level $a$ by that time.
\end{proof}
For the next few statements let us borrow the following result from \cite[Proposition~1.1]{aurzada} (that itself refers to \cite{simon}):

\begin{proposition}\label{proposition:no-exit-in-bdd-time}
Let $\overline{\vert X \vert}$ be the running supremum process of the absolute value process of $X$. Then for some $M\in (0,\infty)$ and then all $\epsilon\in (0,\infty)$ $\PP(\overline{\vert X\vert}_M< \epsilon)>0$,  if and only if
\begin{quote}
$\int 1\land \vert x\vert \measure(dx)<\infty$ and $(\drift=0)\lor (\drift>0\land 0\in \supp(\measure\vert_{(-\infty,0]}))\lor (\drift <0\land 0\in \supp(\measure\vert_{[0,\infty)}))$; or $\diffusion>0$; or $\int 1\land\vert x\vert\measure(dx)=\infty$.
\end{quote}
When so, then $\PP(\overline{\vert X\vert}_M< \epsilon)>0$ whenever $\{M,\epsilon\}\subset (0,\infty)$.
\end{proposition}

\begin{proposition}\label{proposition:proper-exit-after}
The condition that $\law_{a,b}^+[m,\infty)\land \law_{a,b}^-[m,\infty)>0$ whenever $\{a,b\}\subset (0,\infty)$ and $m\in [0,\infty)$, is equivalent to 
\begin{quote}
$\int 1\land \vert x\vert \measure(dx)<\infty$ and $((\drift=0)\land \nu\text{ charges }(0,\infty)\text{ and }(-\infty,0)\text{ both})\lor ((\drift>0)
\land (0\in \supp(\measure\vert_{(-\infty,0]})))\lor ((\drift <0)
\land(0\in \supp(\measure\vert_{[0,\infty)})))$; or $\diffusion>0$; or $\int 1\land\vert x\vert\measure(dx)=\infty$.
\end{quote}
\end{proposition}

\begin{proof}
Since $\law_{a,b}^+[m,\infty)\land \law_{a,b}^-[m,\infty)>0$ whenever $\{a,b\}\subset (0,\infty)$ and $m\in (0,\infty)$, implies $\PP(\overline{\vert X\vert}_M< \epsilon)>0$ for all $\epsilon\in (0,\infty)$ and $M\in (0,\infty)$, necessity of the condition is clear from Propositions~\ref{proposition:non-one-sided-exit} and~\ref{proposition:no-exit-in-bdd-time}.

Sufficiency. Let $\{a,b\}\subset (0,\infty)$, $m\in [0,\infty)$. Let $A$ be the event that the process $X$ will stay within the annulus $(-b,a)$ up to and including time $m$. According to Proposition~\ref{proposition:no-exit-in-bdd-time}, $\PP(A)>0$. Let $B_a$ (respectively $B_b$) be the event that the incremental process of $X$ after $m$ will exit $(-b-X_m,a-X_m)$ in finite time at a level that is $\geq a-X_m$ (respectively $\leq -b-X_m$). By the Markov property, from Proposition~\ref{proposition:non-one-sided-exit}, and via a standard manipulation of conditional expectations, we obtain that $\PP(B_a\vert \FF_m)>0$ a.s. on $A$, hence $\PP(A\cap B_a)=\EE[\PP(B_a\vert \FF_m);A]>0$ and likewise for $A\cap B_b$. 
\end{proof}
Again the situation when $X$ satisfies the condition of Proposition~\ref{proposition:non-one-sided-exit} but not that of Proposition~\ref{proposition:proper-exit-after} can  (up to the trivial transformation $X\rightarrow -X$) rather easily be made more precise (in particular, in Proposition~\ref{proposition:proper-exit-after}, we cannot change the qualification ``whenever $\{a,b\}\subset (0,\infty)$ and $m\in [0,\infty)$'' to ``for some (then all)  $\{a,b\}\subset (0,\infty)$, $m\in [0,\infty)$'' or even to ``for all $m\in [0,\infty)$, for some (then all)  $\{a,b\}\subset (0,\infty)$''):
\begin{proposition}\label{proposition:more-precise-two}
Suppose $\int 1\land\vert x\vert\measure(dx)<\infty$, $\diffusion=0$,  $\drift>0$, $0\notin \supp(\measure\vert_{(-\infty,0]})$, 
but $\nu(-\infty,0)\ne 0$. Let  $W:=-\sup \supp(\nu\vert_{(-\infty,0]})$.  Let furthermore $m\in [0,\infty)$, $\{a,b\}\subset (0,\infty)$. Then $\law^+_{a,b}[m,\infty)\land \law^-_{a,b}[m,\infty)>0$, if and only if $m\gamma_0 < a$ or $a+b>W$.
\end{proposition}
\begin{proof}
Note $\mathbbm{1}_{(-\infty,0)}\cdot \nu$ is finite.

The condition is necessary, since its falsity necessitates $X$ a.s. never exiting $(-b,a)$ for the first time, after and inclusive of time $m$, at the lower boundary.

Sufficiency. Under the assumed condition, with a positive probability, $X$ will not exit $(-b,a)$ up to time (inclusive of) $m$, which fact we can argue as follows. 

Assume first $a+b> W$. According to the L\'evy-It\^o decomposition we can decompose $X$ into the independent sum of a pure drift process with drift coefficient $\gamma_0$, a pure-jump subordinator $Y$ and a compound Poisson process $Z$ with L\'evy measure $\mathbbm{1}_{(-\infty,0)}\cdot \nu$. Let $\{\delta,\gamma\}\in (0,\infty)$, $K\in \mathbb{N}$. With a positive probability $Y$ will not increase by more than $\delta$ by time (inclusive of) $m$. Independently $Z$ will not have a jump during the time interval $[0,a/\gamma_0-\gamma]$, and then $\lceil m/(W/\gamma_0)\rceil$ successive times will behave as follows: have precisely one jump into $(-W-\delta,-W)$ during the next (left-open, right-closed) interval of length $\gamma/K$, and then will not jump for the next (left-open, right-closed) interval of time of length $W/\gamma_0$. It is clear that thanks to $W<a+b$, $\gamma$ can be chosen small enough, and then $\delta$ small enough and $K$ large enough, that on the described event of positive probability $X$ remains in $(-b,a)$ up to time (inclusive of) $m$.



Now assume $m\gamma_0<a$. With a positive probability $X$ will have no negative jump up to time $m$ inclusive, whilst independently and with a positive probability the part of $X$ consisting of the positive jumps of $X$ only, will be found in $[0,a-m\gamma_0)$ at time $m$. 

In either case $X$ remains inside the annulus  $(-b,a)$ up to time (inclusive of) $m$ with a positive probability. Moreover, by the Markov property and from Proposition~\ref{proposition:non-one-sided-exit}, on this event, conditionally on $\mathcal{F}_m$, the incremental process of $X$ after $m$ will exit $(-b-X_m,a-X_m)$ in finite time with a positive probability at a level $\geq a-X_m$ and with a positive probability also at a level $\leq -b-X_m$ (cf. proof of Proposition~\ref{proposition:proper-exit-after}). This demonstrates that in fact $\law^+_{a,b}[m,\infty)\land \law^-_{a,b}[m,\infty)>0$.
\end{proof}

\begin{corollary}\label{corollary:proper-exit-full-support}
The condition that $\law_{a,b}^+[m,M)\land \law_{a,b}^-[m,M)>0$, whenever $\{a,b\}\subset (0,\infty)$, $m\in [0,\infty)$, $M\in (0,+\infty]$, $m<M$, is equivalent to 
\begin{quote}
$\int 1\land \vert x\vert \measure(dx)<\infty$ and $\nu\text{ charges }(0,\infty)\text{ and }(-\infty,0)\text{ both}$ and $(\drift=0)\lor ((\drift>0)
\land (0\in \supp(\measure\vert_{(-\infty,0]})))\lor ((\drift <0)\land (0\in \supp(\measure\vert_{[0,\infty)})))$; or $\diffusion>0$; or $\int 1\land\vert x\vert\measure(dx)=\infty$.
\end{quote}
\end{corollary}
\begin{proof}
One combines Propositions~\ref{proposition:proper-exit-before} and~\ref{proposition:no-exit-in-bdd-time}, using the Markov property in the sufficiency part (similarly as it was used in the proof of Proposition~\ref{proposition:proper-exit-after}). 
\end{proof}
It does not appear the situation in which $X$ satisfies the condition of Proposition~\ref{proposition:non-one-sided-exit} but not of Corollary~\ref{corollary:proper-exit-full-support} can be readily described (in a concise manner).

\begin{remark}
One verifies at once that no two sets of conditions, characterizing the situations present in  Propositions~\ref{proposition:non-one-sided-exit},~\ref{proposition:proper-exit-before}, and~\ref{proposition:proper-exit-after}, and in Corollary~\ref{corollary:proper-exit-full-support} are equivalent in general.
\end{remark}
\bibliographystyle{amsplain}
\bibliography{Biblio_proper}

\providecommand{\bysame}{\leavevmode\hbox to3em{\hrulefill}\thinspace}
\providecommand{\MR}{\relax\ifhmode\unskip\space\fi MR }
\providecommand{\MRhref}[2]{%
  \href{http://www.ams.org/mathscinet-getitem?mr=#1}{#2}
}
\providecommand{\href}[2]{#2}
\begin{thebibliography}{10}

\bibitem{applebaum}
D.~Applebaum, \emph{L{\'e}vy processes and stochastic calculus}, Cambridge
  Studies in Advanced Mathematics, Cambridge University Press, Cambridge, 2009.

\bibitem{aurzada}
F.~Aurzada and S.~Dereich, \emph{{Small deviations of general L\'evy
  processes}}, The Annals of Probability \textbf{37} (2009), no.~5, 2066--2092.

\bibitem{bertoin}
J.~Bertoin, \emph{{L\'e}vy processes}, Cambridge Tracts in Mathematics,
  Cambridge University Press, Cambridge, 1996.

\bibitem{chaumont}
L.~Chaumont, \emph{{On the law of the supremum of L\' evy processes}}, The
  Annals of Probability \textbf{41} (2013), no.~3A, 1191--1217.

\bibitem{doney}
R.~A. Doney and A.~E. Kyprianou, \emph{{Overshoots and undershoots of L\' evy
  processes}}, The Annals of Applied Probability \textbf{16} (2006), no.~1,
  91--106.

\bibitem{kandakov}
V.~Kandakov, T.~Kadankova, N.~Veraverbeke, and N.~Kartaskov, \emph{{Two-sided
  exit problem for general L\'evy processes. Theory and application.}}, Tech.
  report, IAP Statistics Network.

\bibitem{kyprianou}
A.~E. Kyprianou, \emph{{Introductory Lectures on Fluctuations of {L\'e}vy
  Processes with Applications}}, Springer-Verlag, Berlin Heidelberg, 2006.

\bibitem{pecherskii_rogozin}
E.~A. Pecherskii and B.~A. Rogozin, \emph{On joint distributions of random
  variables associated with fluctuations of a process with independent
  increments}, Theory of Probability \& Its Applications \textbf{14} (1969),
  no.~3, 410--423.

\bibitem{rogozin}
B.~A. Rogozin, \emph{On distributions of functionals related to boundary
  problems for processes with independent increments}, Theory of Probability \&
  Its Applications \textbf{11} (1966), no.~4, 580--591.

\bibitem{sato}
K.~I. Sato, \emph{{L\'e}vy processes and infinitely divisible distributions},
  Cambridge studies in advanced mathematics, Cambridge University Press,
  Cambridge, 1999.

\bibitem{simon}
T.~Simon, \emph{{Sur les petites d\'eviations d'un processus de L\'evy}},
  Potential Analysis \textbf{14} (2001), no.~2, 155--173.

\end{thebibliography}
\end{document}